\documentclass[11pt,reqno]{amsart}
\usepackage{amsmath,amsfonts,amssymb,mathrsfs,graphicx}
\newtheorem{theorem}{Theorem}[section]
\newtheorem{lemma}{Lemma}[section]

\newtheorem{Prop}{Proposition}[section]
\newtheorem{remark}{Remark}[section]

\newcommand {\bbr}{\mathbb{R}}
\newcommand {\R}{\mathbb{R}}
\newcommand {\la}{\lambda}
\newcommand {\ga}{\gamma}
\newcommand{\gam}{\gamma}
\newcommand {\al}{\alpha}

\newcommand {\del}{\delta}

\newcommand {\too}{|x|^{-2}}

\newcommand {\w}{\widehat}
\newcommand {\ep} {\varepsilon}
\newcommand {\pa}{\partial}

\newcommand {\TR}{[0,T]\times \bbr^n}
\newcommand \tx[2] {L^{#1}_tL^{#2}_x}

\newcommand \pair  {\tx{\frac {4(n-1)}{n}}{\frac {2(n-1)}{n-2}}}

\newcommand \stx[2] {L_t^{{#1}}H_{x}^{s,{#2}}([0,T]\times \bbr^n)}

\newcommand {\intr}{\int_{\Bbb{R}^n}}

\title[$L^2$-critical Hartree equations]
{Global well-posedness for the $L^2$-critical Hartree  equation on $\bbr^n$, $n\ge 3$.}

\author[M.Chae]{Myeongju Chae}

\address{Department of Applied Mathematics, Hankyong National University, Ansong 456-749, Korea}
 \email{mchae@hknu.ac.kr}

 \author[S.Kwon]{Soonsik Kwon}

 \address{Department of Mathematics, Princeton University, Princeton, New Jersey, USA}
 \email{soonsikk@math.princeton.edu}

\date{}
\thanks{}
\thanks{} 

\begin{document}
\begin{abstract}
We consider the initial value problem for the $L^2$-critical defocusing Hartree equation in $\bbr^n$, $n\ge 3$. We show that the problem is globally well posed in $H^s(\bbr^n)$ when $ 1>s> \frac{2(n-2)}{3n-4}$. We use the ``I-method" following
\cite{C-K-S} combined with a local in time Morawetz estimate for the smoothed out solution $I\phi$ as in \cite{CGT1}.
\end{abstract}
\maketitle

\section{\bf{Introduction}}
In this paper we study the initial value problem of the $L^2$-critical defocusing Hartree
equation,
\begin{align}\label{SP}
\begin{cases}
i\partial_t \phi +\frac12 \Delta \phi = (\too\ast |\phi|^2)\phi, \quad x \in \Bbb R^n,~t>0, \\
\phi(x, 0) = \phi_0(x)\in H^s(\bbr^n).
\end{cases}
\end{align}
Here $H^s(\bbr^n)$ denotes the usual inhomogeneous Sobolev space. \eqref{SP} is meaningful in dimension $n\geq 3$,
 where the Hartree potential is locally integrable. The Hartree type
equations arise in atomic and nuclear physics and is related to the ¡°mean-field theory¡± with respect to wave functions describing boson systems. (\cite{EY}, \cite{Sp}) \\
The local well-posedness results for $s\geq 0$ is shown by the Strichartz estimates similarly
as polynomial type NLS. For $s>0$ \eqref{SP} is locally well-posed in the subcritical sense.
More precisely, for any $\phi_0\in H^s(\R^n)$, the lifetime span of the solution depends on
 the norm of the initial data, $\|\phi_0\|_{H^s}$. Whereas, for $s=0$ the lifetime span depends on the profile of the initial data as well.
 The classical solutions to \eqref{SP} enjoy the mass conservation law,
\[ \|\phi(\cdot,t)\|_{L^2(\bbr^n)} = \|\phi_0(\cdot)\|_{L^2(\bbr^n)},\]
and the energy conservation law,
\begin{equation} \label{energy}
E[t]:= \int_{\bbr^n} |\nabla \phi|^2 +(|x|^{-2}\ast |\phi|^2)|\phi|^2 dx.
\end{equation}
When $s\geq1$, the energy conservation law \eqref{energy} together with the subcritical local theory immediately yields the global well-posedness. But when $0\leq s<1$, where the energy could be infinite, the mass conservation law cannot imply the global well-posedness, since in the local theory for  $L^2$ initial data, the lifetime $T=T(\phi_0)$ could go to zero for a fixed $L^2$ norm. The purpose of this paper is to extend the global well-posedness result below the energy norm. Our main theorem is as follows:
\begin{theorem}\label{T}
 Let $n \geq 3$. The initial value problem of \eqref{SP} is
globally well-posed for initial data $\phi_0\in H^s(\Bbb R^n)$ when $\frac{2(n-2)}{3n-4}< s <1$.
\end{theorem}

We use the $I$-method and the interaction Morawetz inequality,
 which were used in several literatures of the same type of results,
 \cite{ChHKY,CGT1,C-K-S,DPTS,DPTS1,MXZ}. The idea of $I$-method,
 which introduced by Colliander \emph{et.al.} \cite{C-K-Sp},
 is to use a smoothing operator $I$ which regularizes a rough solution up to the regularity level of a conservation law by damping high frequency part. In our example, when $\phi \in H^s$ for $s<1$, $E(\phi)$ may not be finite,
   but for a smoothed function $I\phi$, $E(I\phi)$ is finite. Here, one doesn't expect that $E(I\phi)$ is conserved,
   since $I\phi$ is not a solution to \eqref{SP}. But if $I$ operator is close to the identity operator in some sense,
   $I\phi$ is close to a solution and $E(I\phi)$ is \emph{almost} conserved.
   In fact, we control the growth of $E(I\phi)(t)$ in time.

In addition to $I$-method, we use the interaction Morawetz inequality. Colliander et al. introduced  in \cite{C-K-S} a new Morawetz interaction potential for the nonlinear Schr\"{o}dinger equation in three dimension.
\begin{equation}\label{moreht1}
M[\phi(t)]:= \intr |\phi(x,t)|^2 {\Big(
\int_{\bbr^n}\mbox{Im}\left[\bar{\phi}(y,t)\nabla \phi(y,t)
\right]\cdot \frac{y-x}{|y-x|}dy \Big)}.
\end{equation}
This is a generalization of the classical Morawetz potential, which
has been studied in many literatures especially regarding on the
dispersive property of the Schr\"{o}dinger equations
\cite{B3,Gri,Lin}. The above  functional
\eqref{moreht1} generates a new space-time $L^4_{t,x}$ estimate for
the nonlinear Schr\"{o}dinger equation with the relatively general defocusing power
nonlinearity. Incorporating this with the \textit{almost
conservation law}, they proved the scattering
of the equation and relaxed the low regularity assumption  given  in the previous work \cite{C-K-Sp}.\\
In \cite{ChHKY} the authors showed the almost conservation law and  Morawetz interaction potential approach
worked as well with the Hartree equation in dimension $3$. More precisely, when the defocusing Hartree nonlinearity
 is mass supercritical and energy subcritical case, which is
$(|x|^{-\gamma}\ast |\phi|^2 )\phi, \quad  2<\ga<3$,
 the equation is globally well posed in $H^s(\bbr^3)$, $ 1>s > \max(\frac 12, \frac{4(\ga-2)}{3\ga-4})$
 and has scattering as well. In the $H^1(\bbr^3)$ case,
the same result was shown in \cite{G-V} and later the scattering part was  simplified  in
\cite{Na}.

The interaction Morawetz inequality is extended to other dimensions \cite{TVZ,DPTS,CGT1}.
But in the mass critical case, where the admissible norm is critical, the space-time norm grows in time. We follow the similar way to \cite{CGT1,DPTS}. Due to local in time Morawetz inequality we are able to control
\begin{align}\label{premora}
\|\phi\|_{\pair([0,T]\times \bbr^n)}\le T^{\frac{n-2}{4(n-1)}}
  \|\phi_0\|_{L^2_x}^{\frac 12}\|\phi\|^{\frac{n-2}{n-1}}_{L_t^{\infty}\dot{H}^{\frac 12}_x ([0,T]\times \bbr^n)}
  \end{align}
  for an admissible pair $(\frac{4(n-1)}{n}, \frac{2(n-1)}{n-2})$. The same machinery in \cite{DPTS} with the above
  inequality \eqref{premora} would yield the result that the global well-posedness of \eqref{SP} holds when
   $1>s> \max\left(\frac 12, \frac{2(n-2)}{3n-4} \right)$. Since we allow the admissible space-time norm grows in time,
   we do not know whether scattering holds true.
   Note that  the number $\frac{2(n-2)}{3n-4}$ is lower than $\frac 12$ in dimension 3. The restriction $s> \frac 12$ is inevitable if relying
on the inequality \eqref{premora}. In order to remove this restriction, we use the the inequality
\eqref{premora} for the smoothed out solution $I\phi$. This idea was first introduced in \cite{CGT1,DPTS1}.
They showed it still holds true with negligible error. In our case we have (For detail see Lemma~\ref{almost mora})
\begin{align*}
\|I\phi\|_{\pair([0,T]\times \bbr^n)}&\le T^{\frac{n-2}{4(n-1)}}
  \big(\|\phi_0\|_{L^2_x}^{\frac {1}{n-1}}\|I\phi\|^{\frac{n-2}{n-1}}_{L_t^{\infty}\dot{H}^{\frac 12}_x ([0,T]\times \bbr^n)}+
  \|I\phi\|_{L_t^{\infty}\dot{H}^{\frac 12}_x ([0,T]\times \bbr^n)}^{\frac{2n-6}{2n-3}}\big)
  \\ &  \quad + \quad T^{\frac{n-2}{4(n-1)}} \mbox{ Error}.
  \end{align*}

Since $I\phi$ is in $H^1$(in particular in $\dot{H}^{\frac 12}$), $s$ may go below $\frac 12$.
We show that on the time interval where the local well-posedness the error therm is very small.
 At the time we prepare this paper we are informed that Miao \emph{et.al.} \cite{MXZ} use the same idea to remove the restriction $s>\frac 12$ in the result of $\dot{H}^{\frac 12}$-subcritical Hartree equation as an improvement of \cite{ChHKY}.
On the other hand, Miao \emph{et. al.}\cite{MXZ2,MXZ1} studied
the focusing or defocusing $L^2$ critical Hartree equations as well. They established the global well-posedness and scattering for $L^2$ radial initial data and the blow up criterion to the focusing $L^2$ critical Hartree equation in $\R^3$. \\ \indent
 Before we close the introduction, we would like to add some remark on the $L^2$-critical focusing case,
 \begin{align}\label{Fo}
\begin{cases}
i\partial_t \phi +\frac12 \Delta \phi = -(\too\ast |\phi|^2)\phi, \quad x \in \Bbb R^n,~t>0, \\
\phi(x, 0) = \phi_0(x) \in H^s(\bbr^n).
\end{cases}
\end{align}
 Note that the local well-posedness proof in Section 2 equally works for the focusing case. The equation is known to have a ground state solution $Q$, which solves
 \[\Delta Q - Q= -(|x|^{-2} \ast |Q|^2)Q.\]
 The existence  of $Q$ is proven in \cite{MXZ1} with the decisive property of being the sharp constant of  the Gagliardo-Nirenberg inequality such as
 \[ \int_{\bbr^n} (|x|^{-2}\ast |u|^2)|u|^2(x) dx \le  \frac {2}{\|Q\|_{L^2}}\|u\|_{L^2}^2\|\nabla u\|_{L^2}^2.\] The uniqueness is open except $n=4$, which was settled in  \cite{KLR} adapting E. Lieb's uniqueness proof in \cite{Lieb}.
%
%
%

The paper is organized as follows. In Section~\ref{sec-2}, we review the local well-posedness theorem using the Strichartz estimate.
In Section~\ref{s-4} we give the definition of $I$ operator, show the modified local well-posedness of $I\phi$, and obtain the upper bound of time increment of the modified energy. In Section~\ref{sec-3} we recall the \emph{almost} interaction Morawetz inequality for $I\phi$ and show the error bound. In Section~\ref{sec-5} we conclude the proof of global well-posedness in Theorem~\ref{T}.
\subsection*{Notations}
Given $A,B$, we write $A\lesssim B$ to mean that for some universal constant $K>2$, $A\le
K\cdot B$. We write $A \sim B$ when both $A \lesssim B$ and
$B\lesssim A$. The notation $A\ll B$ denotes $B> 3\cdot A$. We write
$\langle A \rangle \equiv (1+A^2)^{\frac 12}$, and $\langle \nabla
\rangle$ for the operator with Fourier multiplier $(1 +
|\xi|^2)^{\frac 12}$. The symbol $\nabla$ denote the spatial gradient.
We will often use the notation $\frac 12 + \equiv \frac 12
+\epsilon$ for some universal $0< \ep \ll 1$. Similarly, we write
$\frac 12- \equiv \frac 12 -\ep$. We use the function space
$L_t^qL_x^r$ and $H^{s,p}$ given norms by
\begin{align*}
 \|F\|_{L_t^qL_x^r(\bbr^{n+1})} &\equiv
\left( \int_{\bbr}\left(\int_{\bbr^n}|F(x,t)|^r dx \right)^{\frac qr} dt\right)^{\frac 1q}, \\
\|u\|_{H^{s,p}(\bbr^n)} & \equiv \|\mathcal{F}^{-1}[(1+|\xi|^2)^{\frac s2}\mathcal{F}u]\|_{L^p(\bbr^n)},
\end{align*}
where $\mathcal{F}$ is a fourier transform,  $1\le p,q,r \le \infty$.

\subsection*{Acknowledgements} M.C. is supported by KRF-2007-C00020. S.K. thanks Terry Tao for helpful conversations.

\section{\bf{The local well-posedness}}\label{sec-2}
We refer $(q,r)$ the admissible pair when $2\le q <\infty$, $2\le r\le \frac{2n}{n-2}$ and
\[ \frac 2q + \frac{n}{r} = \frac n2 \]
and state the Strichartz inequality in dimension $n$.
\begin{Prop}\label{prop2.1}
Suppose that $(q,r)$, $(\la, \eta)$ are any two admissible pairs.
Suppose that $u(x,t)$ is a solution of the problem
\begin{align}\label{inhomo}
i\pa_tu(x,t)+\Delta u(x,t) = F(x,t),~~(x,t) \in \bbr^n\times [0,T],
\end{align}
for a data $u(0)\in H^s$, $F\in \stx {\la'}{\eta'}(\TR)$ where $\la'$
and $\eta'$ are the H\"{o}lder conjugates of $\la$ and $\eta$,
respectively. Then $u$ belongs to $\stx{q}{r}\cap
C_tH_x^{s,r}([0,T]\times \bbr^3)$ and we have the estimate
\begin{equation*}\label{stri}
\|u\|_{\stx{q}{r}}\lesssim \|u(0)\|_{H^s(\bbr^n)}+
\|F\|_{\stx{\la'}{\eta'}}.
\end{equation*}
\end{Prop}
For the pure power nonlinearity $\la |u|^{\al}u$, the local well-posedness of
$i\pa_t u +\frac 12 \Delta u = \la |u|^{\al}u$ with the rough data $u(0)\in H^s$, $ 0<s<1$
was proven in \cite{Ca} (See also \cite{Ca1,Tao book}).

We define the Strichartz norm of functions $\phi:[0,T]\times \bbr^n \to \mathbb{C}$ by
\[\|\phi\|_{S^0_T} = \displaystyle\sup_{(q,r)\, admissible} \|\phi\|_{\tx{q}{r}([0,T]\times \bbr^n)}.\]
In particular  $S^0_T \subset C_tL_x^2([0,T]\times \bbr^n)$.
Then the Strihartz estimates may be written as
\[ \|\phi\|_{S^0_T} \le \|\phi\|_{L^2} + \|(i\pa_t +\Delta)\phi\|_{\tx{q'}{r'}([0,T]\times \bbr^n)},\]
where $(q',r')$ is the conjugate of an admissible pair $(q,r)$.

The local existence theorem  of \eqref{SP} is as follows.
\begin{theorem}\label{thm2.1}
For a given $\phi_0\in H^s(\bbr^n)$, $0 <s $,
there exists a positive time $T=T(\|\phi_0\|_{H^s})$ and the unique solution $\phi$ of (\ref{SP}),
in $\phi\in C_tH_x^s([0,T]\times \bbr^n) \cap S_T^s$ for every admissible pair $(q,r)$, where
\[\|\phi\|_{S^s_T} = \displaystyle\sup_{(q,r)\, admissible}\|\langle\nabla\rangle^s \phi\|_{\tx{q}{r}([0,T]\times \bbr^n)}.\]
\end{theorem}
\begin{proof}
Let $S^L(t)$ be the flow map $e^{it\Delta}$ corresponding to the the linear Schr\"{o}dinger equation.
Then the Duhamel formulation of \eqref{SP} is
\[\phi(t)= S^L(t)\phi_0 -i \int_0^t S^L(t-\tau)|x|^{-2}\ast |\phi|^2 \phi(\tau) d\tau.\]
We will show that the map
$A: \phi \longrightarrow S^L(t)\phi_0 -i \int_0^t
S^L(t-\tau)[(\too\ast |\phi|^2) \phi](\tau) d\tau$ is a contraction mapping on the ball
$\|\phi\|_{S^s_T} \le 2M$ when   $T$ is chosen later and  $\|\phi_0\|_{H^s}<M $.

Let us show $A$ is well defined on $X$. Applying the linear and the dual Strichartz estimates, we have
\begin{align}\label{aphi}
\|A\phi\|_{S^s_T} \lesssim
 \|\phi_0\|_{H^s} +\| |x|^{-2}\ast |\phi|^2\phi \|_{\stx {\la'}{\eta'}}
 \end{align}

for any admissible $(\la,\eta)$.
We recall the Leibnitz rule for
fractional Sobolev spaces \cite{CW,Tay}:
\noindent For $s>0,$ $1<p<\infty$,
\[ \|fg\|_{H^{s,p}} \lesssim \|f\|_{L^{q_1}}\|g\|_{H^{s,q_2}}+ \|f\|_{L^{r_1}}\|g\|_{H^{s,r_2}}\]
provided $\frac 1p= \frac{1}{q_1} + \frac{1}{q_2}= \frac{1}{r_1}+
\frac{1}{r_2}, \mbox{ with }
\, q_2,r_2\in (1,\infty) \mbox{ and } q_1,r_1\in (1,\infty].$ \\

Let us choose $(\la', \eta')= (\frac{4}{3+s}, \frac{2n}{n-s+1})$.
The fractional Leibnitz rule, Hardy-Sobolev inequality and H\"{o}lder's inequality lead to
\begin{align}\label{V}\begin{aligned}
\| (\too\ast |\phi|^2) \phi\|_{H^{s,\frac{2n}{n-s+1}}} &\le \|\too\ast
|\phi|^2\|_{H^{s,n}}\|\phi\|_{L^{\frac{2n}{n-s-1}}}
 + \|\too\ast |\phi|^2\|_{L^{\frac{n}{1-s}}}\|\phi\|_{{H^{s,\frac{2n}{n+s-1}}}}\\
&\lesssim \||\phi|^2\|_{H^{s,\frac{n}{n-1}}} \|\phi\|_{L^{\frac{2n}{n-s-1}}}
 + \|\phi\|^2_{L^{\frac{2n}{n-s-1}}} \|\phi\|_{H^{s,\frac{2n}{n-s+1}}}\\
&\lesssim 2 \|\phi\|^2_{L^{\frac{2n}{n-s-1}}} \|\phi\|_{H^{s,\frac{2n}{n+s-1}}}.\\
\end{aligned}\end{align}
By the Sobolev embedding we have
\[ \| (\too\ast |\phi|^2) \phi\|_{H^{s,\frac{2n}{n-s+1}}} \lesssim \|\phi\|_{H^{s,\frac{2n}{n+s-1}}}^3.\]
Combining this with \eqref{aphi} we find
\begin{align*}
\| A\phi\|_{S^s_T} &\lesssim \|\phi_0\|_{H^s} +
\left(\int_0^T \|\phi\|^{\frac{12}{3+s}}_{H^{s,\frac{2n}{n+s-1}}} dt\right)^{{\frac{3+s}{4}}}\\
& \lesssim  \|\phi_0\|_{H^s}+  T^s\|\phi\|^{3}_{\stx {\frac{4}{1-s}}{\frac{2n}{n+s-1}}}\\
& \lesssim \|\phi_0\|_{H^s} + T^s \|\phi\|_{S^s_T}^3.
\end{align*}
The local well-posedness time $T$ is chosen as $T\lesssim \|\phi_0\|_{H^s}^{-\frac 2s}$.
Similarly, one can show that $A$ is a contraction. And uniqueness assertion and continuous dependence on data follow in the same manner.
\end{proof}

\section{\bf{Almost conservation law of the modified energy}}\label{s-4}
In this section, we define the smoothing operator $I_N$, which sends an $H^s$ function to an $H^1$ function. We find a bound of the growth of $E(I_N\phi)(t)$ in time. \\
The operator $I_N$ is defined as in \cite{C-K-S}. Let $N \gg 1$ be a parameter to be chosen later. Define
\begin{equation}\label{hl}
\widehat{I_Nf}(\xi) \equiv m(\xi)\widehat f(\xi),
\end{equation}
where the multiplier $m(\xi)$ is smooth, radially symmetric,
nonincreasing in $|\xi|$ and satisfies
\begin{equation}\label{m}
m(\xi)=
\begin{cases}
\begin{array}{ll}
1 & |\xi|\le N \\
\left( \frac {N}{|\xi|}\right)^{1-s} & |\xi|\ge 2N .
\end{array}
\end{cases}
\end{equation}
We note that $m(\xi)$ satisfies the H\"{o}rmander multiplier
condition. As intended, the definition of $m(\xi)$ gives the
following relations between $\|I_N\phi\|_{H^1}$ and $\|\phi\|_{H^s}$
for $0<s<1$;
\begin{align}
\|I_N\phi\|_{H^1(\bbr^n)}&\approx \sum_{k\le \log N}(1+2^k)\|P_k \phi\|_{L^2(\bbr^n)} +
\sum_{k > \log N}N^{1-s}(1+2^{k})^s\|P_k \phi\|_{L^2(\bbr^n)}\nonumber\\
&\lesssim N^{1-s}\|\phi\|_{H^s(\bbr^3)}\label{Iphi}\\
\|\phi\|_{ H^s(\bbr^n)} &\lesssim \sum_{k\le \log N}(1+2^k)^s \|P_k I\phi\|_{L^2(\bbr^n)} +
\sum_{k> \log N}(1+2^k)N^{s-1}\|P_k I\phi\|_{L^2(\bbr^n)}\nonumber\\
& \lesssim \|I_N\phi\|_{H^1(\bbr^n)}, \nonumber
\end{align}
where $P_k \phi$ is defined by $\widehat {P_k \phi}(\xi)=
\varphi(\xi/2^k) \widehat \phi(\xi)$ for a nonnegative smooth function
$\varphi$ with $\mbox{supp } \phi =\{ \xi| 2^{-1}\le |\xi|\le 2 \}$
and $\sum_{k\in \mathbb Z}\varphi(2^{-k}\xi)=1$. What it follows  we write $I$ for $I_N$ suppressing $N$.\\
Let us define the iteration space $Z_{I}(t)$ as
\begin{eqnarray*} 
Z_{I}(t) \ = \ \sup_{(q,r)\,admissible}
\|\langle\nabla\rangle I\phi\|_{L_t^qL_x^r([0,t]\times \bbr^3)},
\end{eqnarray*}

\subsection{Modified local theory}
First of all, we prove a local well-posedness result for the modofied solution $I\phi$. This theorem is essentially similar to the local well-posedness proof at the critical regularity in \cite{Ca}. But here we assume critical Strichartz norm of $I\phi$ is small, instead of $\phi$. Similar proofs are found in \cite{CGT1}, \cite{DPTS1}.

\begin{lemma}\label{lem4.1}
For given initial data $\phi_0\in H^s(\R^n)$ for $0<s$, there are time $T^*>0$ and a universal constant $\delta>0$ satisfying the following:
\begin{enumerate}
\item The solution $\phi(x,t)$ to \eqref{SP} exists on $[0,T^*]\times
\bbr^n$,
\item If \[ \|I\phi\|_{\pair([0,T^*]\times \bbr^n)} \le \del,\]
then \[ Z_I(T^*) \ \lesssim \|\langle\nabla\rangle I \phi_0\|_{L^2(\bbr^n)}.\]
\end{enumerate}
\end{lemma}

\begin{proof} \text{ }\\
The first part is from the local well-posedness theorem, Theorem~\ref{thm2.1}. The second part is also done by the Strichartz estimate \eqref{stri} in the Duhamel formula with $\langle\nabla\rangle I$ operator:
$$ \langle\nabla\rangle I\phi(x,t) =  S^L(t)\langle\nabla\rangle I\phi_0 -i \int_0^t S^L(t-\tau)\langle\nabla\rangle I(|x|^{-2}\ast |\phi|^2 \phi(\tau)) d\tau. $$
For all $0\le t \le T^*$, we have
\begin{align}\label{es Z}
Z_I(t) &\lesssim \|I\phi_0\|_{H^1} + \|\langle\nabla\rangle I((\too\ast |\phi|^2)\phi)\|_{L_t^{\gam'} L_x^{\rho'}}\\
& \lesssim \|I\phi_0\|_{H^1} + \|(\too\ast \langle\nabla\rangle I|\phi|^2 )\phi\|_{L_t^{\gam'} L_x^{\rho'}}+
\|((\too\ast |\phi|^2)\langle\nabla\rangle I\phi)\|_{L_t^{\gam'} L_x^{\rho'}},  \nonumber
\end{align}
where $(\gamma,\rho)$ is admissible. In the previous step we have used Leibniz's rule for $\langle\nabla\rangle I$. Note that in the high frequency $(|\xi|>N)$, $I$ is a negative derivative, but $\langle\nabla\rangle I$ is a positive fractional derivative. A simple modification of the proof of the fractional Leibniz rule works for it.
Let us  choose $(\gamma, \rho)= (4, \frac{2n}{n-1})$.
In fact we can use any admissible pair satisfying  $\gamma \ge \frac {2(n-1)}{n-2}$.
\noindent We first estimate $\|(\too\ast |\phi|^2)\nabla
I\phi\|_{\tx{\frac 43}{\frac{2n}{n+1}}}$. By using
H\"{o}lder's, fractional Sobolev's  inequalities, we obtain
\begin{align}\label{s-1}\begin{aligned}
\|(\too\ast |\phi|^2)\langle\nabla\rangle I\phi \|_{\tx{\frac 43}{{\frac{n+1}{2n}}}}
&\le \|\too\ast |\phi|^2\|_{\tx{q_1}{r_1}}
\|\langle\nabla\rangle I\phi\|_{\tx{q_2}{r_2}}\\
& \le\|\phi\|^2_{\tx{2q_1}{2p}}\|\langle\nabla\rangle I\phi\|_{\tx{q_2}{r_2}}\\
&\le\|\phi\|^2_{\tx{2q_1}{2p}}Z_I,
\end{aligned}\end{align}
where
\[ \frac 34 = \frac{1}{q_1}+ \frac{1}{q_2}, \quad \frac{2n}{n+1}= \frac{1}{r_1}+ \frac{1}{r_2},\quad
1+\frac{1}{r_1}=\frac 2n +\frac{1}{p}\]
and $(q_2, r_2)$ is admissible.  Due to  scaling argument, $(2q_1, 2p)$ is expected to be admissible. Let $(2q_1, 2p)= (\frac {4(n-1)}{n}, \frac {2(n-1)}{n-2})$, then $(q_2, r_2) = (\frac{4(n-1)}{n-3},\frac{2n(n-1)}{n^2-2n+3})$. \\
In a similar way, the other term is also estimated as follows:
\begin{align}\label{s-11}
\|(\too\ast \langle\nabla\rangle I|\phi|^2 )\phi\|_{L_t^{\gam'} L_x^{\rho'}} \leq \|\phi\|^2_{\pair}Z_I
\end{align}
Now we estimate $\|\phi\|^2_{\pair}$. We decompose $\phi$ into its frequency localized pieces, $\phi= P_{\le N}\phi+  \sum_{j=1}^{\infty} P_{N_j}\phi$, where $N_j=2^{k_j}$ and and $k_j$'s are consecutive integers starting from $[\log N]$ indexed by $j=1,2,3\cdots$. By triangle inequality we get
\begin{align}\label{pair es}\begin{aligned}
\|\phi\|_{\pair} &\le \|P_{\le N}\phi\|_{\pair} + \sum_{j=1}^{\infty}\|P_{N_j}\phi\|_{\pair} \\
&= \|P_{\le N}\phi\|_{\pair} + \sum_{j=1}^{\infty}\|P_{N_j}\phi\|^{\epsilon}_{\pair}\|P_{N_j}\phi\|^{1-\epsilon}_{\pair}.
\end{aligned}\end{align}
From the definition of $I$ operator we have the followings:
\begin{align*}
\|P_{\le N}\phi\|_{\pair} &= \|I\phi\|_{\pair}, \\
\|P_{N_j}\phi\|_{\pair} &\lesssim N_j^{1-s}N^{s-1}\|IP_{N_j}\phi\|_{\pair}, \\
\|P_{N_j}\phi\|_{\pair} &\lesssim N_j^{-s}N^{s-1}\|\langle\nabla\rangle IP_{N_j}\phi\|_{\pair}. \\
\end{align*}
Putting these together into \eqref{pair es}, we obtain
\begin{align*}
\|\phi\|_{\pair} &\lesssim \|P_{\le N}\phi\|_{\pair} + \\ &\sum_{j=1}^{\infty}N_j^{-s+\epsilon}N^{s-1}\|P_{N_j}\phi\|^{\epsilon}_{\pair}\|\langle\nabla\rangle IP_{N_j}\phi\|^{1-\epsilon}_{\pair}.
\end{align*}
Ignoring $N^{s-1} \leq 1$ and using the fact that $ \|P_N f\|_{L^p} \lesssim \|f\|_{L^p}$, one can sum up over $j$, if $s>\epsilon$. Thus, we have 
$$ \|\phi\|_{\pair} \lesssim \|P_{\le N}\phi\|_{\pair} + \|I\phi\|^{\epsilon}_{\pair}\|\langle\nabla\rangle I\phi\|^{1-\epsilon}_{\pair}.$$
Hence, from \eqref{es Z} we conclude
$$ Z_I \lesssim \|I\phi_0\|_{H^1} + Z_I\|I\phi\|^2_{\pair} + Z_I^{3-2\epsilon}\|I\phi\|^{\epsilon}_{\pair}. $$
Choosing sufficiently small $\delta$ and $T^*$, we conclude the proof.
\end{proof}

\subsection{Almost conservation law} We show the \emph{almost conservation law} of the modified energy. \\
The usual energy \eqref{energy} is shown to be conserved by differentiating in time
\begin{align*}
&\frac{d}{dt}E(\phi)(t)\\
&\mbox =  \int_{\bbr^n} 2\text{Re}\overline{\pa_t \phi}(2(|x|^{-2}\ast |\phi|^2)\phi - \Delta\phi-2\pa_t\phi)
 + (|x|^{-2}\ast\pa_t|\phi|^2) |\phi|^2 - (|x|^{-2}\ast|\phi|^2)\pa_t|\phi|^2 dx \\
&\mbox = \int_{\bbr^n}(|x|^{-2}\ast\pa_t|\phi|^2) |\phi|^2 - (|x|^{-2}\ast|\phi|^2)\pa_t|\phi|^2 dx\\
& \mbox = 0 ,
\end{align*}
 using the equation \eqref{SP}. Since $I\phi$ is not a solution to the equation \eqref{SP}, $E(I\phi)(t)$ is not conserved. But still we have a control of the time increment of the modified energy $E(I\phi)(t)$. Differentiating
$E(I\phi)(t)$ in time, we obtain
\begin{align*}
\frac{d}{dt}E(I\phi)(t)  = & \int_{\bbr^n}2\text{Re}
\overline{\pa_t I\phi}\left[2(I(|x|^{-2}\ast |\phi|^2)\phi)- \Delta I\phi
- 2i \pa_t I\phi\right]dx.
\end{align*}
Then we have
\begin{align}
E(I\phi(T))-E(I\phi(0))& = \ 4\text{Re} \int_0^T \int_{\bbr^n}
\overline{\pa_t I\phi}
\left[ (|x|^{-2}\ast |I\phi|^2)I\phi- I((|x|^{-2}\ast|\phi|^2)\phi)\right]\ dx \, dt \nonumber\\
& := \ E_T (t)\quad \label{p}
\end{align}
The following proposition shows that $E(I\phi)$ is an \emph{almost} conserved quantity.
\begin{Prop}\label{ACL}
Assume we have $s > 0$, $N\gg 1$, $\phi_0 \in
C_0^{\infty}(\bbr^n)$, and a solution of \eqref{SP} on a time
interval $[0,T]$ for which
\[ \|I\phi\|_{\pair(\TR)} \lesssim \del.\]
Assume in addition that $\| \langle \nabla \rangle I\phi_0\| \lesssim 1$. Then we conclude that for all $t\in [0,T]$,
\[ E(I\phi)(t)= E(I\phi_0) + O(N^{-{1+}}).\]
\end{Prop}
\begin{proof}[Proof of Proposition~\ref{ACL}] \text{ } \\
We compute in the frequency space. Applying the Parseval formula to $E_T$ in \eqref{p}, we obtain
\begin{eqnarray}\label{E1}
 E_T
&=& \text{Re} \int_0^T\!\!\!\int_{\sum_{j=1}^4 \xi_j=0} \left(1-
\frac{m(\xi_2+\xi_3+\xi_4)}{m(\xi_2)m(\xi_3)m(\xi_4) }\right)\nonumber \\
& & \quad \times \ |\xi_2+\xi_3|^{-(n-2)}\, \w{ \overline {\pa_t
I\phi}}(\xi_1)\, \w{\overline{I\phi}}(\xi_2) \, \w{I\phi}(\xi_3)
\, \w{I\phi(\xi_4)}\ d\xi_1 \, d\xi_2 \, d\xi_3 \, d\xi_4\, dt.
\end{eqnarray}
\noindent Now if we use equation (\ref{SP}) to substitute for
$\pa_t I\phi$ in (\ref{E1}), then it is split into two terms
 as follows:
\begin{eqnarray*}
 E_a &\equiv& \bigg|
 \int_0^T\!\!\!\int_{\sum_{j=1}^4 \xi_j=0} \left(1-
\frac{m(\xi_2+\xi_3+\xi_4)}{m(\xi_2)m(\xi_3)m(\xi_4) }\right)\,  |\xi_2+\xi_3|^{-(n-2)} \\
& & \quad \times \ \w{\Delta \overline { I\phi}}(\xi_1)\,
\w{\overline{I\phi}}(\xi_2) \, \w{I\phi}(\xi_3) \,
\w{I\phi}(\xi_4)\ d\xi_1 \, d\xi_2 \, d\xi_3 \,
d\xi_4 \, dt\, \bigg|,\nonumber\\
E_b &\equiv& \bigg| \int_0^T\!\!\!\int_{\sum_{j=1}^4 \xi_j=0}
\left(1-
\frac{m(\xi_2+\xi_3+\xi_4)}{m(\xi_2)m(\xi_3)m(\xi_4) }\right)\,  |\xi_2+\xi_3|^{-(n-2)} \\
& & \quad \times \ \w{\overline { (I(\too *
|\phi|^{2}\phi))}}(\xi_1)\, \w{\overline{I\phi}}(\xi_2) \,
\w{I\phi}(\xi_3) \, \w{I\phi}(\xi_4) \ d\xi_1 \, d\xi_2 \, d\xi_3
\, d\xi_4\, dt \, \bigg|.\nonumber
\end{eqnarray*}

In both cases, we break down $\phi$ into Littlewood-Paley pieces $\phi_j$,
each localized in $2^{k_j}$ in frequency, $\langle\xi_j\rangle\sim2^{k_j}=N_j$, $k_j=0,1,2,\cdots$,
and then use a version of Coifman-Meyer estimate for a class of multiplier operators.  \\
\begin{Prop}[Proposition  $6.1$ in \cite{ChHKY}]\label{multiplier}
  Let $\sigma(\xi)$ be infinitely
differentiable so that for all $\alpha \in N^{nk}$ and all $\xi=
(\xi_1,\dots, \xi_k) \in \bbr^{nk}$. Then there is a constant
$c(\alpha)$ with
\begin{equation}\label{cm}
|\pa_{\xi}^{\alpha}\sigma(\xi)| \le c(\alpha)\,
(1+|\xi|)^{-|\alpha|}.
\end{equation}
Let the multi-linear operator $\Lambda$ be given
\begin{align*}
[\Lambda(f_1, \cdots, f_k)](x) &= \int_{\bbr^{nk}} e^{ix(\xi_1+\cdots+\xi_k)}
\sigma(\xi_1,\cdots,\xi_k) \,   |\xi_2+\xi_3|^{-(n-2)} \, \hat
f_1(\xi_1)\,
\hat f_2(\xi_2) \cdots \hat f_k(\xi_k) \ d\xi_1 \, \cdots \,d\xi_k
\end{align*}
 for $k\ge 2$.
  Then we have
  \[ \|\Lambda(f_1, \cdots, f_k)\|_{L^p} \lesssim \|f_1\|_{L^{p_1}}\|f_2\|_{L^{p_2}} \cdots \| f_k \|_{L^{p_k}}\]
  where $(p,p_i)$ is related by $\frac 1p+1= \frac 2n +\sum_{i=1}^k \frac{1}{p_i}$.
\end{Prop}
We first estimate a pointwise bound on the symbol
$$ \left| 1- \frac{m(\xi_2+\xi_3+\xi_4)}{m(\xi_2)m(\xi_3)m(\xi_4)}\right| \leq B(N_2,N_3,N_4) $$
Factoring $B(N_1,N_2,N_3)$ out of the integral in $E_T$, it leaves a symbol $\sigma_1$,
 which satisfies the condition of Proposition~\ref{multiplier}, as the following:
\begin{align*}
 \sum_{N_1,N_2,N_3,N_4}& B(N_2,N_3,N_4) \int_0^T\int_{\R^n} [\Lambda(\Delta I\phi_1, I\phi_2,I\phi_3)]\widehat{}(\xi_4)\widehat{I\phi_4}(\xi_4)d\xi_4dt \\
+& \sum_{N_1,N_2,N_3,N_4} B(N_2,N_3,N_4) \int_0^T\int_{\R^n} [\Lambda(IP_{N_1}(|x|^{-2}\ast|\phi|^2\phi), I\phi_2,I\phi_3)]\widehat{}(\xi_4)\widehat{I\phi_4}(\xi_4)d\xi_4dt
\end{align*}
where $$ [\Lambda(f,g,h)](x) = \int_{\bbr^{3n}} e^{ix(\xi_1+\xi_2+\xi_3)}
\sigma_1(\xi_1,\xi_2,\xi_3) \, |\xi_2+\xi_3|^{-(n-2)} \, \hat
f(\xi_1)\,
\hat g(\xi_2) \, \hat h(\xi_3) \ d\xi_1 \, d\xi_2 \,d\xi_3 $$ and
$$ \sigma_1(\xi_1,\xi_2,\xi_3)= 1-\frac{m(\xi_1)}{m(\xi_2)m(\xi_3)m(\xi_1+\xi_2+\xi_3)} \Big{/} \Big| 1-\frac{m(\xi_1)}{m(\xi_2)m(\xi_3)m(\xi_1+\xi_2+\xi_3)}  \Big|.$$

We shall show that
$$ E_a + E_b \lesssim N^{-1+}(Z_I(T))^P $$
for some $P>0$. \\
For this aim, we claim that
\begin{align}
 & \sum _{N_1, N_2, N_3, N_4} \int_0^T\!\!\!\int_{\bbr^n} B(N_2,N_3,N_4) \big[ \,
\Lambda (\Delta
I\phi_1,I\phi_2,I\phi_3)\, \big] (x,t) \, I\phi_4(x,t) \ dx \, dt \label{f1}\\
& \ + \sum _{N_1, N_2, N_3, N_4}  \int_0^T\!\!\!\int_{\bbr^n} B(N_2,N_3,N_4)
\big[ \, \Lambda (IP_{N_1}(\too \ast |\phi|^{2}\phi),I\phi_2,I\phi_3) \,
\big](x,t) \, I\phi_4(x,t)
\ dx \, dt \label{f2}\\
&\lesssim \ N^{-1+} \, (Z_{I}(T)^{4} + Z_{I}(T)^{6}). \nonumber
\end{align}

From Proposition~\ref{multiplier}, we have
\begin{equation}\label{multilinear1}
\|\Lambda(f,g,h)\|_{L^p} \lesssim \|f\|_{L^{p_1}}\|g\|_{L^{p_2}}\|h\|_{L^{p_3}}
\end{equation}
where $ \frac 1p = \frac 2n -1 +\frac 1{p_1} +\frac 1{p_2} + \frac 1{p_3} $.\\
For the first term $E_a$, we use \eqref{multilinear1} and H\"older inequality to get

\begin{align}
\label{form E1a}
\Big|\int_0^T \int_{\bbr^n} B(N_2,N_3,N_4) \big[
\Lambda &(\Delta I\phi_1,I\phi_2,I\phi_3) \big] (x,t)  I\phi_4(x,t)  dx  dt
\Big| \\
\lesssim &\|\Delta I\phi_1\|_{L^{q_1}_tL^{p_1}_x} \|I\phi_2\|_{L^{q_2}_tL^{p_2}_x}\|I\phi_3\|_{L^{q_3}_tL^{p_3}_x} \|I\phi_4\|_{L^{q_4}_tL^{p_4}_x}  \nonumber
\end{align}

where $\sum \frac 1{p_i} + \frac 2n -1 = 1$, $\sum \frac 1{q_1} =1$.
Choosing $ \frac 1{p_i} = \frac {n-1}{2n}$, $\frac 1{q_i} = \frac 14$, and using Bernstein inequlity, we obtain
$$ \eqref{f1} \lesssim B(N_2,N_3,N_4) \frac{N_1}{N_2N_3N_4}(Z_I(T))^4.  $$

We reduce to show
\begin{equation}\label{N-1}
\sum_{N_1,N_2,N_3,N_4} B(N_2,N_3,N_4) \frac{N_1}{N_2N_3N_4} \lesssim N^{-1+\epsilon}
\end{equation}
By symmetry we may assume $N_2\geq N_3\geq N_4$. Then it suffices to consider the following three cases.\\
\textbf{Case 1}: $N\gg N_2$. We have $m(\xi_i) =1 $ since $\sum_i \xi_i =0$. So, the symbol
$$ \left| 1- \frac{m(\xi_1)}{m(\xi_2)m(\xi_3)m(\xi_4)} \right| = 0. $$
\textbf{Case 2}: $N_2 \geq N \gg N_3 \geq N_4$. Since $\sum_i\xi_i = 0$, we have $N_1 \sim N_2$. By the mean value theorem,
\begin{align*}
\Big|1- \frac{m(\xi_1)}{m(\xi_2)m(\xi_3)m(\xi_4)} \Big| &= \Big| \frac{m(\xi_2)-m(\xi_2+\xi_3+\xi_4)}{m(\xi_2)}  \Big| \\
 & \lesssim \frac{|\nabla m(\xi_2)\cdot (\xi_3+\xi_4)|}{m(\xi_2)} \lesssim \frac{N_3}{N_2}.
\end{align*}
Thus, 
\begin{align*}
 \eqref{f1} &\lesssim \frac{1}{N_2N_4}(Z_I(T))^4 \\
  & \lesssim N^{-1+\epsilon}N_2^{-\epsilon}(Z_I(T))^4
\end{align*}
Summing up with $N_4,N_3,N_2$, we have \eqref{N-1}. \\
\textbf{Case 3}: $N_2\geq N_3\gtrsim N$. In this case we need to consider two subcases $N_1\sim N_2$ and $N_2 \gg N_1$ since by $\sum_i\xi_i = 0$ the case $ N_1 \gg N_2$ cannot happen. \\
For the first case, $N_1\sim N_2$, we estimate
\begin{align*}
  \Big|1- \frac{m(\xi_1)}{m(\xi_2)m(\xi_3)m(\xi_4)}  \Big|\frac{N_1}{N_2N_3N_4}  &\lesssim \frac{1}{N_3m(\xi_3)N_4m(\xi_4)}\\
   = \frac{1}{N_3(\frac{N}{|\xi_3|})^{1-s}N_4m(\xi_4)} &\sim \frac {N^s}{N_3^s}\cdot\frac{1}{N}\frac{1}{N_4m(\xi_4)} \\
  & \lesssim N^{-1+\epsilon}N_3^{-\epsilon}
\end{align*}
since $x m(x) \geq 1$ for $x \geq 1$. We can sum up $N_4,N_3$ directly. But when summing up $N_2$, we use the Cauchy-Schwartz inequality with $\phi_i=P_{N_i}I\phi$ as follows:
$$\sum_N P_N\nabla I\phi\cdot P_N\nabla I\phi  \leq \big(\sum_N(P_N\nabla I\phi)^2 \big).$$
In the second case, $N_2 \gg N_1$, again by $\sum_i \xi_i =0$, we have $N_2\sim N_3$.
\begin{align*}
\Big|1- \frac{m(\xi_2+\xi_3+\xi_4)}{m(\xi_2)m(\xi_3)m(\xi_4)} \Big|&\frac{N_1}{N_2N_3N_4} \lesssim \frac{m(\xi_1)}{m(\xi_2)^2m(\xi_4)}\frac{N_1}{N_2N_3N_4} \\
 &\sim N_1m(\xi_1)\frac{1}{N^2}\frac{N^{2s}}{N_2^{2s}}\cdot \frac{1}{N_4m(\xi_4)}
\end{align*}
For our purpose, we want to show
$$ \Big(\frac{N}{N_2}\Big)^{2s} N_1m(\xi_1) \lesssim N. $$
If $N_1 \leq N$, then $m(\xi_1) =1$ and this is true. If $N_1 \gtrsim N $, then
$$ \Big(\frac{N}{N_2}\Big)^{2s} N_1m(\xi_1) = \frac{N^{1+s}}{N_2^s}\cdot \Big(\frac{N_1}{N_2}\Big)^s = N\frac{(NN_1)^s}{N_2^{2s}} \leq N. $$
This conclude the proof of \eqref{f1}.
Now we turn to the estimate of $E_b$. The above analysis is applied to $E_b$, once we show the following lemma.
\begin{lemma}
\begin{equation}
  \|P_{M}I([|x|^{-2}\ast (\phi_1\phi_2)]\phi_3)\|_{L^4_tL^{\frac{2n}{n-1}}_x}  \lesssim M (Z_I(T))^3
\end{equation}
\end{lemma}
\begin{proof}
We divide $\phi$ into $\phi= \phi_{lo} + \phi_{hi}$ where
\begin{align*}
  &\text{supp}\, \widehat{\phi}_{lo}(\xi,t) \subseteq \{|\xi| < 2\} \\
  &\text{supp}\, \widehat{\phi}_{hi}(\xi,t) \subseteq \{|\xi| > 1\}
\end{align*}
In the case that all $\phi$'s are $\phi_{lo}$ we simply estimate
\begin{align*}
  \|P_M\big(I[(|x|^{-2}&\ast\phi_{lo}\bar{\phi}_{lo})\phi_{lo}]   \big)  \|_{L^4_tL^{\frac{2n}{n-1}}_x} =   \|(|x|^{-2}\ast\phi_{lo}\bar{\phi}_{lo})\phi_{lo} \|_{L^4_tL^{\frac{2n}{n-1}}_x}^3  \\
  & \lesssim \|\phi_{lo}\|_{L^{12}_tL^{\frac{6n}{3n-5}}_x}^3 \\
  & \lesssim \|\langle\nabla\rangle I\phi_{lo}\|_{L^{12}_tL^{\frac{6n}{3n-1}}_x}^3 \\
  & \lesssim M(Z_I(T))^3.
\end{align*}
When all $\phi$'s are $\phi_{hi}$, we use Bernstein inequality, Sobolev embedding and the Leibniz rule as following:
\begin{align*}
\|\frac 1M &P_M\big(I[(|x|^{-2}\ast\phi_{hi}\bar{\phi}_{hi})\phi_{hi}] \big)  \|_{L^4_tL^{\frac{2n}{n-1}}_x} \\
  & \lesssim \|\nabla^{-1}P_M\big(I[(|x|^{-2}\ast\phi_{hi}\bar{\phi}_{hi})\phi_{hi}] \big)  \|_{L^4_tL^{\frac{2n}{n-1}}_x}  \\
  & \lesssim \|\langle\nabla\rangle^{\frac{n}{3n-4}}I[(|x|^{-2}\ast\phi_{hi}\bar{\phi}_{hi})\phi_{hi}]\|_{L^4_tL^{\frac{3n^2+n-4}{(3n-4)2n}}_x} \\
  & \lesssim \|\langle\nabla\rangle^{\frac{n}{3n-4}}I\phi_{hi} \|^3_{L^{12}_tL^p_x} \\
  & \text{         where     } {\footnotesize \frac 3p= -\frac 2n + 1 +\frac{3n^2+1-4}{(3n-4)2n} } \\
  & \lesssim \|\langle\nabla\rangle^{1}I\phi_{hi} \|^3_{L^{12}_tL^{\frac{6n}{3n-1}}_x} \\
  & \lesssim (Z_I(T))^3
\end{align*}

where we used $ \frac{1}{n}-\frac{1}{3n-4} \geq -\frac 1p + \frac{3n-1}{6n}$.\\
The remaining $lo-hi$ cases are controlled in a similar manner to the $hi-hi$ case. We omit the detail here.
\end{proof}
Hence, we have shown \eqref{f1}, \eqref{f2} and so conclude the proof.
\end{proof}

\section{\bf{Almost interaction Morawetz estimate in $\R^n$, $n\geq 3$}}\label{sec-3}
In this section, we show the \emph{almost} interaction Morawetz inequality.
Let us start by recalling the higher dimensional interaction Morawetz inequality for a general nonlinearity. The interaction Morawetz inequality was developed in \cite{C-K-S} in $\R^3$ and this higher dimensional extension was derived in \cite{TVZ}. We first recall higher dimensional interaction Morawetz inequality for a general nonlinearity.

\begin{lemma}[\cite{TVZ}, Proposition 5.5]
Let $\phi$ solve
$$ i\partial_t\phi + \frac 12\Delta\phi = \mathcal{N} $$
on $I\times \R^n$. Assume that $ \text{Im}(\mathcal{N}\overline{\phi}) =0 $. \\
Then, we have
\begin{align}\label{general mora}
 - \int_I \iint_{\R^n\times\R^n} & \Delta\big(\frac {1}{|y-x|}\big)|\phi(x,t)|^2|\phi(y,t)|^2dxdydt \\
& +2\int_I\iint_{\R^n\times\R^n}|\phi(x,t)|^2\frac{y-x}{|y-x|}\cdot\{\mathcal{N},\phi\}(y,t) dxdydt  \nonumber \\
& \lesssim \|\phi_0\|^2_{L^2_x}\|\phi\|^2_{L^\infty_t\dot{H}^{\frac 12}([0,T]\times \bbr^n)}  \nonumber
\end{align}
where $\{f,g\} = \text{Re}(f\nabla \overline{g} - g\nabla \overline{f}) $.
\end{lemma}

First, we apply \eqref{general mora} to the solution to \eqref{SP},
where $\mathcal{N} = (|x|^{-2}\ast |\phi|^2)\phi$. A computation shows that the second term is positive. 
\begin{align*}
\text{the second term of \eqref{general mora}} &=
-2 \int |\phi(x,t)|^2\frac{y-x}{|y-x|}\cdot \nabla_y\frac{1}{|y-z|^2}|\phi(z,t)|^2|\phi(y,t)|^2 dxdydz  \\
& =4\int |\phi(x,t)|^2|\phi(y,t)|^2|\phi(z,t)|^2 \frac{y-x}{|y-x|}\cdot\frac{y-z}{|y-z|^4} dxdydz \\
& =2\int |\phi(x,t)|^2|\phi(y,t)|^2|\phi(z,t)|^2 \big(\frac{y-x}{|y-x|}-\frac{z-x}{|z-x|}\big)\cdot\frac{y-z}{|y-z|^3} dxdydz \\
& \geq 0
\end{align*}

By the same analysis as in \cite{TVZ}, we obtain several estimates of space-time $L^q_tL^p_x$-norms.
\begin{Prop}
  Let $\phi(t,x)$ be a classical solution to \eqref{SP}. Then we have \\
  when $n=3$,
  \begin{equation}\label{n=3 L^4}
    \|\phi\|_{L^4_{t,x}(\R\times\R^3)} \lesssim \|\phi_0\|^{\frac 12}_{L^2_x(\R^3)}\|\phi(t)\|^{\frac 12}_{L^\infty_t\dot{H}^{\frac 12}_x(\R\times\R^3)}
  \end{equation}
  when $n\geq 4$,
 \begin{equation}
 \||\nabla|^{-\frac{n-3}{4}}\phi\|_{L^4_{t,x}([0,T]\times\R^n)} \lesssim \|\phi_0\|^{\frac 12}_{L^2}\|\phi(t)\|^{\frac 12}_{L^\infty_t\dot{H}^{\frac 12}_x([0,T]\times \bbr^n)}
 \end{equation}
  \begin{align}
  &\|\phi\|_{L^{2(n-1)}_tL^{\frac{2(n-1)}{n-2}}_x([0,T]\times \bbr^n)}\le \|\phi_0\|_{L^2_x}^{\frac 12}\|\phi\|^{\frac{n-2}{n-1}}_{L_t^{\infty}\dot{H}^{\frac 12}_x ([0,T]\times \bbr^n)} \\
  \label{mora}
  &\|\phi\|_{\pair([0,T]\times \bbr^n)}\le T^{\frac{n-2}{4(n-1)}}
  \|\phi_0\|_{L^2_x}^{\frac 12}\|\phi\|^{\frac{n-2}{n-1}}_{L_t^{\infty}\dot{H}^{\frac 12}_x ([0,T]\times \bbr^n)}.
\end{align}
\end{Prop}

\begin{proof}
A detailed proof is found in \cite{TVZ}, Section 5. Here we give a sketch. \\
In dimension $n=3$, we have formally $-\Delta\frac{1}{|x|} = 4\pi\delta$, and then \eqref{n=3 L^4} follows. \\
In higher dimension, $n\geq 4$, we obtain $ -\Delta\frac{1}{|x|}= \frac{n-3}{|x|^3}$. A convolution with $\frac{1}{|x|^3}$ is essentially to take the fractional derivative $|\nabla|^{-(n-3)}$. Hence we obtain from \eqref{general mora}
$$ \|\nabla^{-\frac{n-3}{2}}|\phi|^2\|_{L^2_{t,x}([0,T]\times \bbr^n)} \lesssim \|\phi_0\|^{\frac 12}_{L^2_x(\R^n)}\|\phi\|^{\frac 12}_{L^\infty_t\dot{H}^{\frac 12}([0,T]\times\R^n)}. $$
From Lemma 5.6 in \cite{TVZ}

\begin{equation}\label{mora1}
 \||\nabla|^{-\frac{n-3}{4}}\phi\|_{L^4_{t,x}([0,T]\times \bbr^n)} \lesssim \|\phi_0\|^{\frac 12}_{L^2_x(\R^n)}\|\phi\|^{\frac 12}_{L^\infty_t\dot{H}^{\frac 12}([0,T]\times\R^n)}.
\end{equation}
Interpolation between \eqref{mora1} and the trivial estimate
$$ \||\nabla^{\frac 12}\phi\|_{L^\infty_tL^2_x} \leq \|\phi\|_{L^\infty_t\dot{H}^{\frac 12}_x}  $$
and using the H\"older's inequality in time we have

\begin{align}\label{preI}
  &\|\phi\|_{\pair([0,T]\times \bbr^n)}\le T^{\frac{n-2}{4(n-1)}}\|\phi_0\|_{L^2_x}^{\frac 12}\|\phi\|^{\frac{n-2}{n-1}}_{L_t^{\infty}\dot{H}^{\frac 12}_x ([0,T]\times \bbr^n)}.
\end{align}

\end{proof}

For the initial data below $\dot{H}^{\frac 12}$, the above estimate is not useful since $\dot{H}^{\frac 12}$-norm of the solution may not be finite. To overcome this difficulty, we use the interaction Morawetz inequality into the smoothed solution $I\phi$. Write the $I$-Hartree equation as the following:
\begin{align*}
 iI\phi_t& + \frac 12 \Delta I\phi \\
 &= (|x|^{-2}\ast I\phi\overline{I\phi})I\phi + \big[I((|x|^{-2}\ast \phi\overline{\phi})\phi)  - (|x|^{-2}\ast I\phi\overline{I\phi})I\phi \big]  \\
 &=: \mathcal{N}_{good} + \mathcal{N}_{bad}.
\end{align*}
Then using \eqref{general mora} we obtain
\begin{align}\label{I mora}
\begin{aligned}
 - \int_I \iint_{\R^n\times\R^n} & \Delta\big(\frac {1}{|y-x|}\big)|I\phi(x,t)|^2|I\phi(y,t)|^2dxdydt \\
& + 2\int_I\iint_{\R^n\times\R^n}|I\phi(x,t)|^2\frac{y-x}{|y-x|}\cdot\{\mathcal{N}_{good},I\phi\}(y,t) dxdydt \\
& + 2\int_I\iint_{\R^n\times\R^n}|I\phi(x,t)|^2\frac{y-x}{|y-x|}\cdot\{\mathcal{N}_{bad},I\phi\}(y,t) dxdydt \\
& \lesssim \|I\phi\|^2_{L^\infty_tL^2_x}\|I\phi\|^2_{L^\infty_t\dot{H}^{\frac 12}([0,T]\times \bbr^n)} \\
& + \int_0^T\iint_{\bbr^n\times\bbr^n} |\mbox{Im}\,(\mathcal N_{bad}\overline{I\phi}(t, y)) \nabla(I\phi(t, x)) I\phi(t,x)|dxdydt
\end{aligned}
\end{align}
By the same computation as above, one can see the second term of \eqref{I mora} is positive. We wish the third term involving $\mathcal{N}_{bad}$ to be small.
Similarly to \eqref{preI} we have
\begin{align}\begin{aligned}\label{Error}
\| I\phi&\|_{\pair([0,T]\times \bbr^n)}\\
& \lesssim T^{\frac{n-2}{4(n-1)}}\left(    \|I\phi\|_{L_t^\infty L^2_x}^{\frac{1}{n-1}}\|I\phi\|_{L_t^{\infty}\dot{H}^{\frac 12}_x }^{\frac{n-2}{n-1}} + \|I\phi\|_{L_t^{\infty}\dot{H}^{\frac 12}_x }^{\frac{2n-6}{2n-3}} +\mbox{  Error}\right),
\end{aligned}\end{align}  where Error is defined in the lemma below.
\begin{lemma}\label{almost mora}
On a time interval $J$ where the local well-posedness in Theorem~\ref{thm2.1} holds true, we have that
\begin{align}\label{error estimate}
 \text{Error} = &\int_J\iint_{\R^n\times\R^n}|I\phi(x,t)|^2\frac{y-x}{|y-x|}\cdot\{\mathcal{N}_{bad},I\phi\}(y,t) dxdydt \\
  & + \int_J\iint_{\bbr^n\times\bbr^n} |\mbox{Im}\,(\mathcal N_{bad}\overline{I\phi}(t, y)) \nabla(I\phi(t, x)) I\phi(t,x)|dxdydt \nonumber\\& \lesssim \frac 1{N^{1-}}Z_I(J)^6 \nonumber
\end{align}
In particular, if we assume $\| \langle \nabla \rangle  I\phi_0\|_{L^2}\lesssim 1 $ and
$\| I\phi\|_{\pair(J\times \bbr^n)} \le \del$, then $$ Error \lesssim  \frac {1}{N^{1-}}.$$
\end{lemma}
\begin{proof}
We rewrite the error term via $\mathcal{N}_{bad} = I((|x|^{-2}\ast\phi\overline{\phi})\phi)- (|x|^{-2}\ast I\phi I\overline{\phi})I\phi$:
\begin{align*}
  \text{Error} &=  \int_J\iint_{\R^n\times\R^n}|I\phi(x,t)|^2\frac{y-x}{|y-x|}\cdot \big(\mathcal{N}_{bad}\nabla \overline{I\phi} - I\phi\nabla\overline{\mathcal{N}_{bad}}\big)(y,t) dxdydt \\
  & + \int_J\iint_{\bbr^n\times\bbr^n} |\mbox{Im}\,(\mathcal N_{bad}\overline{I\phi}(t, y)) \nabla(I\phi(t, x)) I\phi(t,x)|dxdydt \\
  & \leq \int_J\int_{\R^n} |\mathcal{N}_{bad}|\cdot|\nabla I\phi|dydt\|I\phi\|^2_{L^\infty_JL^2_x}  +
  \int_J\int_{\R^n} |\nabla\mathcal{N}_{bad}|\cdot|I\phi|dydt\|I\phi\|^2_{L^\infty_JL^2_x}  \\
  & + \| \mathcal{N}_{bad}\|_{L^1_JL^2_x}\| I\phi\|^2_{L^\infty_JL^2_x}\| \langle\nabla\rangle I\phi\|_{L^\infty_JL^2_x} \\
  &\lesssim  \|\langle\nabla\rangle\mathcal{N}_{bad}\|_{L^1_JL^2_y}\|\langle\nabla\rangle I\phi\|_{L^\infty_JL^2_y}\|I\phi\|^2_{L^\infty_JL^2_x} \\
  &\lesssim \|\langle\nabla\rangle\big[I((|y|^{-2}\ast\phi\overline{\phi})\phi)- (|y|^{-2}\ast I\phi I\overline{\phi})I\phi \big]\|_{L^1_tL^2_y}(Z_I(J))^3
\end{align*}
We reduce to show
\begin{equation}\label{3.18}
\|\langle\nabla\rangle\big[I((|x|^{-2}\ast\phi\overline{\phi})\phi)- (|x|^{-2}\ast I\phi I\overline{\phi})I\phi \big]\|_{L^1_tL^2_x} \lesssim \frac{1}{N^{1-}}(Z_I(J))^3.
\end{equation}
By Plancerel theorem in space, we have
\begin{align*}
\mathcal{F}\,&\nabla\big[I((|x|^{-2}\ast\phi\overline{\phi})\phi)- (|x|^{-2}\ast I\phi I\overline{\phi})I\phi  \big](-\xi_1)  \\
& = \int_{\sum_{i=1}^4\xi_i=0} i\xi_1\big[\frac{m(\xi_1)-m(\xi_2)m(\xi_3)m(\xi_4)}{m(\xi_2)m(\xi_3)m(\xi_4)}\big]
\widehat{I\phi}(\xi_2)\widehat{I\phi}(\xi_3)\widehat{I\phi}(\xi_4)|\xi_2+\xi_3|^{-(n-2)}d\xi_2d\xi_3d\xi_4,
\end{align*}
where we ignored complex conjugates since they don't make any differences. As we did in Section~\ref{s-4}, we decompose $\phi$ into a sum of dyadic pieces. It is reduced to show
\begin{align*}
&\sum_{N_2,N_3,N_4} \big\|\int_{\xi_i\sim N_i,i=2,3,4}\sigma(\xi_2,\xi_3,\xi_4)\widehat{I\phi}(\xi_2)\widehat{I\phi}(\xi_3)\widehat{I\phi}(\xi_4)|\xi_2+\xi_3|^{-(n-2)}d\xi_2d\xi_3d\xi_4    \big\|_{L^1_tL^2_{\xi_1}}  \\
&\lesssim \sum_{N_2,N_3,N_4} \big\|\int_{\xi_i\sim N_i,i=2,3,4}\frac{1}{\xi_2\xi_3\xi_4}\sigma(\xi_2,\xi_3,\xi_4)\widehat{\nabla I\phi}(\xi_2)\widehat{\nabla I\phi}(\xi_3)\widehat{\nabla I\phi}(\xi_4)|\xi_2+\xi_3|^{-(n-2)}d\xi_2d\xi_3d\xi_4    \big\|_{L^1_tL^2_{\xi_1}}                      \\
&\lesssim \frac{1}{N^{1-}} (Z_I(T))^3
\end{align*}
where $$ \sigma(\xi_2,\xi_3,\xi_4) = |\xi_1| \frac{m(\xi_1)- m(\xi_2)m(\xi_3)m(\xi_4)}{m(\xi_2)m(\xi_3)m(\xi_4)} $$

We use Proposition~\ref{multiplier}. Note that the exponent numerology
$ 3\cdot\frac{3n-4}{6n} = \frac 12 + 1 -\frac 2n$ and that $(3,\frac{6n}{3n-4})$ is admissible. Thus, once we show that
\begin{equation}\label{symbol estimate}
\frac{N}{|\xi_2\xi_3\xi_4|}\sigma(\xi_2,\xi_3,\xi_4) \lesssim 1,
\end{equation}
then we have
\begin{align*}
\big\|\int_{\xi_i\sim N_i,i=2,3,4}&\frac{1}{\xi_2\xi_3\xi_4}
\sigma(\xi_2,\xi_3,\xi_4)\widehat{\nabla I\phi}(\xi_2)\widehat{\nabla I\phi}(\xi_2)\widehat{\nabla I\phi}(\xi_2)|\xi_2+\xi_3|^{-(n-2)}d\xi_2d\xi_3d\xi_4    \big\|_{L^1_tL^2_{\xi_1}} \\
&\lesssim \frac 1N (Z_I(T))^3.
\end{align*}

The proof of \eqref{symbol estimate} is very similar to the proof of Proposition~\ref{ACL}. So, a sketch is enough. We assume $N_2\geq N_3\geq N_4$ by symmetry and consider the following cases. \\
\textbf{Case 1:} $N \gg N_2$. The symbol is identically zero. \\
\textbf{Case 2:} $N_2 \geq N \gg N_3 \geq N_4$. Since $\sum_i\xi_i = 0$, we have $N_1 \sim N_2$. By the mean value theorem, we estimate
\begin{align*}
\Big|1- \frac{m(\xi_1)}{m(\xi_2)m(\xi_3)m(\xi_4)} \Big| &= \Big| \frac{m(\xi_2)-m(\xi_2+\xi_3+\xi_4)}{m(\xi_2)}  \Big| \\
 & \lesssim \frac{|\nabla m(\xi_2)\cdot (\xi_3+\xi_4)|}{m(\xi_2)} \lesssim \frac{N_3}{N_2}.
\end{align*}
Thus,
$$ \frac{N}{\xi_2\xi_3\xi_4}\sigma(\xi_2,\xi_3,\xi_4) \lesssim \frac{N}{N_2N_3N_4}N_1\frac{N_3}{N_2} \lesssim 1 $$

\textbf{Case 3:} $N_2\geq N_3\gtrsim N$. In this case we need to consider two subcases $N_1\sim N_2$ and $N_2 \gg N_1$ due to $\sum_i\xi_i = 0$. \\
For the first case, $N_1\sim N_2$, we estimate
\begin{align*}
  \frac{NN_1}{N_2N_3N_4}\Big|1- \frac{m(\xi_1)}{m(\xi_2)m(\xi_3)m(\xi_4)}  \Big| &\lesssim \frac{N}{N_3m(\xi_3)N_4m(\xi_4)}\\
   = \frac{N}{N_3(\frac{N}{N_3})^{1-s}N_4m(\xi_4)} &\sim \frac {N^s}{N_3^s}\cdot\frac{1}{N}\frac{1}{N_4m(\xi_4)} \\
  & \lesssim 1,
\end{align*}
where used $x m(x) \geq 1$ for $x \geq 1$. \\
In the second case, $N_2 \gg N_1$, again by $\sum_i \xi_i =0$, we have $N_2\sim N_3$.
\begin{align*}
\Big|1- \frac{m(\xi_2+\xi_3+\xi_4)}{m(\xi_2)m(\xi_3)m(\xi_4)} \Big|\frac{NN_1}{N_2N_3N_4} &\lesssim \frac{m(\xi_1)}{m(\xi_2)^2m(\xi_4)}\frac{NN_1}{N_2N_3N_4} \\
 &\sim N_1m(\xi_1)\frac{1}{N^2}\frac{N^{2s}}{N_2^{2s}}\cdot \frac{N}{N_4m(\xi_4)}
\end{align*}
For our purpose we want to show
$$ \Big(\frac{N}{N_2}\Big)^{2s} N_1m(\xi_1)\frac 1N \lesssim 1. $$
If $N_1 \leq N$, then $m(\xi_1) =1$ and
$$        \Big(\frac{N}{N_2}\Big)^{2s} \frac{N_1}{N} \lesssim 1.   $$
If $N_1 \gtrsim N $, then
$$ \Big(\frac{N}{N_2}\Big)^{2s} N_1m(\xi_1)\frac 1N = \frac{N^{s}}{N_2^s}\cdot \Big(\frac{N_1}{N_2}\Big)^s \lesssim 1. $$
\end{proof}

\section{\textbf{Proof of Main Theorem}}\label{sec-5}
We combine the interaction Morawetz estimate and Propotion 4.1 with a scaling argument
to prove the following statement
giving a uniform bound in terms of the $H^s$-norm of the initial data.
\begin{Prop}
 Suppose $\phi(x,t)$ is a global in time solution to \eqref{SP} from data $\phi_0\in C_0^{\infty}(\bbr^n)$.
 Then for a given large $T$  we have
 \begin{align}
  \|\phi(T)\|_{H^s} \lesssim_{\|\phi_0\|_{H^s}} T^{\alpha(s,n)}
 \end{align}
 as long as $ \frac{2(n-2)}{(3n-4)} <s<1$. The positive number $\alpha(s,n)$ depends on $s$ and $n$.
\end{Prop}
\begin{remark}
Since $T$ is arbitrarily large, the a priori bound on the $H^s$ norm gives the global well-posedness in the range of
$ \frac{2(n-2)}{(3n-4)}<s<1$.
\end{remark}
\begin{proof}
The equation \eqref{SP} is invariant over scaling of
 \[\phi^{\la}(x,t)\equiv \la^{-\frac n2} \phi(\frac{x}{\la}, \frac{t}{\la^2}).\]
  According to
 \[ \| \nabla I\phi_0^ {\la}\|^2_{L^2(\bbr^n)}
 \lesssim \left( N^{1-s} \la^{ -s}\|\phi_0\|_{H^s (\bbr^n)}\right)^2,\]
 we choose $\la$ as
\begin{equation}\label{choice}
\la \approx N^{\frac{1-s}{s}}.
\end{equation}
in order to normalize $ \| \nabla I\phi_0^ {\la}\|^2_{L^2(\bbr^n)}\lesssim O(1).$
The second term of  the modified energy $E(I\phi_0^{\la})$ is treated as follows,
\begin{align*}
 \|\too\ast |I\phi_0^{\la}|^2 |I\phi_0^{\la}|^2 \|_{L^1(\bbr^n)}
  & \lesssim \|\too\ast |I\phi_0^{\la}|^2 \|_{L^{n} (\bbr^n)}
 \|I\phi_0^{\la}\|_{L^{\frac{2n}{n-1}} (\bbr^n)}^2 \\
 & \lesssim \|I\phi_0^{\la}\|_{L^{\frac{2n}{n-1}} (\bbr^n)}^4
\lesssim \|I\phi_0^{\la}\|_{L^2(\bbr^n)}^2\|I\phi_0^{\la}\|^2_{L^{\frac{2n}{n-2}}(\bbr^n)}\\
&\lesssim   \|I\phi_0\|_{\dot{H}^{1}(\bbr^n)}^2 \quad \mbox{ by  Sobolev embedding.}
\end{align*}
Hence we have $E(I\phi_0^{\la})\lesssim 1$.
The remaining proof is similar to the proof of Theorem $5.1$ in \cite{CGT1} with necessary modification on exponents.
As we have already seen, Hartree type nonlinearity behaves smoother than the polynomial type $|\phi|^{\frac 4n}\phi$.

Let us pick a time $T_0$ arbitrarily large, and let us define
\[S:= \{ 0<t\le \la^2 T_0: \|I\phi^{\la}\|_{\pair([0,t]\times \bbr^n) } \le K t^{\frac{n-2}{4(n-1)}}\} \]
with $K$, $N$  a constant to be chosen later. We claim that S is the whole interval $[0,\la^2T_0]$. Assuming not, there exists a time $T\in[0,\la^2 T_0)$ so that
\begin{align}\label{boot}
 KT^{\frac{n-2}{4(n-1)}} < \|I\phi^{\la}\|_{\pair([0,T]\times \bbr^n)}<  2KT^{\frac{n-2}{4(n-1)}}
 \end{align} by continuity.

We now split the interval $[0,T]$ into consecutive subintervals $J_k$, $k=1,\cdots,L$ so that
\[\|I\phi^{\la}\|_{\pair(J_k\times \bbr^n)}^{\frac{4(n-1)}{n}} \le \del \]
where $\del$ defined as in Lemma $3.1$. Note that
\[ L \sim \frac{(2K)^{\frac{4(n-1)}{n} } T^{\frac{n-2}{n} }}{\del}\]
due to \eqref{boot}. From Proposition $3.1$ we know that for any $ 0<s<1$
\[ \sup_{[0,T]} E(I\phi^{\la}(t)) \lesssim  E(I\phi_0^{\la})+ LN^{-1+}.\]
Now we fix $N$ such that $LN^{-1+} \lesssim 1$ holds for all $t\in [0,T]$. Since $T<\la^2 T_0$, this is achieved
if provided
\[ \frac{(2K)^{\frac{4(n-1)}{n}} (\la^2 T_0)^{\frac {n-2}{n}}}{\del} \sim N^{1-}.\]
By substituting $ \la = N^{\frac{1-s}{s}}$, the above is equal to
\begin{align}\label{N}
 N^{1- \frac{1-s}{s} \frac{2(n-2)}{n}} \sim \frac{(2K)^{\frac{4(n-1)}{n}}}{\del} T_0^{\frac {n-2}{n}}.
 \end{align}
Thus we choose $N$ as above for arbitrary $T_0$ as long as $\frac {2(n-2)}{3n-4}<s <1$.

On the other hand we have that in Lemma $4.2$,
\begin{align*}\|I\phi^{\la}\|_{\pair([0,T]\times \bbr^n)}^{\frac{4(n-1)}{n} } & \lesssim T^{\frac {n-2}{n}}
(\|I\phi^{\la}\|_{L_t^{\infty}L^2_x(\bbr^n)}^{\frac {4}{n}}\|I\phi^{\la}\|_{L_t^{\infty}\dot{H}_x^{\frac 12}( \bbr^n)}^{\frac {4(n-2)}{n}} +\|I\phi^{\la}\|_{L_t^{\infty}\dot{H}^{\frac 12}_x }^{\frac{2n-6}{2n-3}} )\\
& + T^{\frac {n-2}{n}} \int_0^T \mbox{Error } dt.
\end{align*}
We know that  $\int_{J_k} \mbox{Error } dt \lesssim N^{-1+}$ on each $J_k$. Hence summing up all the ${J_k}'s$, we
find
\[\int_0^T \mbox{Error } dt \lesssim LN^{-1+} \lesssim 1\]
by the choice of $\la$, $N$ as \eqref{choice}, \eqref{N}.
Thus with the trivial bound $\|I\phi\|_{\dot{H}^{\frac 12}}\le \|I\phi\|_{{H}^{1}}$, we have
$\|I\phi^{\la}\|_{\pair([0,T]\times \bbr^n)}^{\frac{4(n-1)}{n}}\lesssim T^{\frac {n-2}{n}}$.
This estimate contradicts \eqref{boot} for a proper choice of $K$.

Therefore, we conclude $S = [0, \la^2 T_0 ]$ and $T_0$ can be arbitrary large.
In addition we also have that for $s > \frac {2(n-2)}{3n-4}$
\[ \| I _N\phi^\la (\la^2 T_0 ) \|_{H_x^1} = O(1), \]
from which we estimate
\begin{align*}
\| \phi (T_0 ) \|_{H^s} & \le \| \phi \|_{L^2} + \la^s \| \phi^\la (\la^2 T_0 ) \|_{\dot H^s} \\
& \lesssim \la^s \| I \phi^\la (\la^2 T_0 ) \|_{H_x^1}
\lesssim \la^s \lesssim N^{1-s} \\
& \lesssim T_0^{\al(s,n)}
\end{align*}
where $\al(s,n)= \frac{(n-2)s(1-s)}{s(3n-4)-2(n-2)}$.
\end{proof}


\begin{thebibliography}{99}

\bibitem{B3} J. Bourgain,
{\it Global solutions of nonlinear Shr\"{o}dinger equations,}
Amer. Math. Soc., Providence, RI, 1999.
\bibitem{Ca} T. Cazenave, F. Weissler,
{\it The cauchy problem for the critical nonlinear
Schr\"{o}dinger equation in $H^s$,} Nonlinear Anal. T.M.A. {\bf 14} (1990), 807-836.
\bibitem{Ca1} T. Cazenave,
{\it Semilinear Schr\"{o}dinger Equations}.
Courant Lecture Notes {\bf 10}, New York (2003).
\bibitem{C-M} R. Coifman and Y. Meyer,
{\it Au del\'{a} des op\'{e}rateurs pseudo-differentiels,}
Ast\'{e}risque, Soci\'{e}t\'{e} Math\'{e}matique de France {\bf 57}, 1978.
\bibitem{ChHKY} M.Chae, S. Hong, J. Kim, C. Yang, {\it Scattering theory below energy for a class of Hartree type,} Comm. pde {\bf 33}(2008), 321-348
\bibitem{CW} M. Christ, M. Weinstein, {\it Dispersion of small amplitude solutions of
the generalized Korteweg-de Vries equation,} J. Fucnt. Anal. {\bf 100} (1991), 87-109.
\bibitem{CGT1} J. Colliander, M. Grillakis, and N. Tzirakis  \emph{Improved interaction Morawetz inequalities for the cubic nonlinear Schroedinger equation in 2d}, Int. Math. Res. Not. {\bf 23} (2007), Art. ID rnm090, 30 pp.
\bibitem{C-K-Sp} J. Colliander, M. Keel, G. Staffilani, H. Takaoka, and T. Tao,
{\it Almost conservation laws and global rough solutions to a Nonlinear Schr\"{o}dinger Equation}, Math. Res. Letters, {\bf 9} 1-24, 2002.
\bibitem{C-K-S} J. Colliander, M. Keel, G. Staffilani, H. Takaoka, and T. Tao,
{\it Global existence and scattering for rough solutions of a
nonlinear Schr\"{o}dinger equation on $\mathbb{R}^3$,} Comm. Pure
Appl. Math. {\bf 57} (2004), 987-1014.
\bibitem{C-K2} J. Colliander, M. Keel, G. Staffilani, H. Takaoka, and T. Tao,
{\it Viriel, Morawetz, and interaction Morawetz inequalities},
preprint, http://www.math.ucla.edu/~tao/preprints/pde.html.
\bibitem{CRSW} J. Colliander, S. Raynor, C. Sulem, and J.D. Wright, {\it Ground state mass concentration in the $L^2$-critical nonlinear Sch\"{o}dinger equation below $H^1$}, Math. Res. Lett. {\bf 12} (2005), 357-375.

\bibitem{DPTS} D. de Silva, N. Pavlovic, G. Staffillani, and N. Tzirakis, \emph{Global Well-Posedness for the $L^2$-critical nonlinear Schroedinger equation in higher dimensions}, Comm. Pure Appl. Anal. {\bf 6} (2007), no. 4, 1023--1041.
\bibitem{DPTS1} D. de Silva, N. Pavlovic, G. Staffillani, and N. Tzirakis, \emph{Global well-posedness and polynomial bounds for the defocusing nonlinear Schroedinger equation in 1d},  Comm. PDE {\bf 33} (2008), no. 7-9, 1395--1429.

\bibitem{EY} L. Erd\"{o}s, H.-T. Yau, \emph{ Derivation of the nonlinear SchrAodinger equation from a many body
Coulomb system}, Adv. Theor. Math. Phys. {\bf 5} no. 6 (2001), 1169-1205.

\bibitem{G-M-P} I. Gasser, P. A. Markowich, and B. Perthame,
{\it Dispersion and Moment Lemmas Revisited,} J. Diff. Equations,
{\bf 156} (1999), 254-281.
\bibitem{Gri} M. Grillakis,
{\it On nonlinear Schr\"{o}dinger equations,} Comm. Partial Differential Equations, {\bf 25}
(2000), 1827-1844.
\bibitem{G-V1} J. Ginibre and G. Velo,
{\it Smoothing properties and retarded estimates for some dispersive evolution equations,}
Comm. Math. Phys. {\bf 123} (1989), 535-573.
\bibitem{G-V} J. Ginibre and G. Velo,
{\it Scattering theory in the energy space for a class of Hartree equations,}
Contemp. Math. {\bf 263} (2000), 29-60.
\bibitem{HK}  T. Hmidi, S. Keraani, {\it Remarks on the blowup for the $L^2$-critical nonlinear Schr\"{o}dinger
equations}, SIAM J. Math. Anal. {\bf 38} (2006), no. 4, 1035-1047.
\bibitem{KLR} J. Krieger, E. Lenzmann, and P. Rapha\"{e}l, {\it On stability of pseudo-conformal blowup for $L^2$-critical equations,} preprint, arXiv:0808.2324.
\bibitem{Lieb} E. Lieb {\it Existence and uniqueness of the minimizing solution of Choquard¡¯s
nonlinear equation,} Studies in Appl. Math. {\bf 57} (1976/77), no. 2, 93.105.
\bibitem{Lin} J. Lin and W. Strauss,
{\it Decay and scattering of solutions of a nonlinear Schr\"{o}dinger equation,}
J. Funct. Anal. {\bf 30} (1978),245-263.

\bibitem{MXZ2} C. Miao, G. Xu, and L. Zhang, \emph{Global well-posedness and scattering for the mass-critical Hartree equation with radial data}, J. Math. Pures Appl. {\bf 91} (2009) 49-79.

\bibitem{MXZ1} C. Miao, G. Xu, and L. Zhang, \emph{On the blow up phenomenon for the $L^2$-critical focusing Hartree equation in $\R^4$}, preprint, arXiv:0708.2614.

\bibitem{MXZ} C. Miao, G. Xu, and L. Zhang, \emph{Global well-posedness and scattering for the defocusing $H^{\frac12}$-subcritical Hartree equation in $\R^d$}, Ann. I. H. Poincare-AN (2009), doi:10.1016/j.anihpc.2009.01.003.
\bibitem{Na} K. Nakanishi, \emph{Energy scattering for Hartree equations}, Math. Res. Lett. {\bf 6} (1999),
107-118.
\bibitem{Sp} H. Spohn, {\it On the Vlasov hierarchy}, Math. Methods Appl. Sci. {\bf 3} (1981), no. 4,  445-455.

\bibitem{Tao book} T. Tao, \emph{Nonlinear Dispersive equations, Local and global analysis.}, CBMS Regional Conference
Series in Mathematics, 106. Published for the Conference Board of the Mathematical
Sciences, Washington, DC; by the American Mathematical Society, Providence, RI, 2006.

\bibitem{TVZ}
T. Tao, M. Visan, and X. Zhang, \emph{The Schrodinger equation with combined power-type nonlinearities}, Comm. PDE {\bf 32} (2007), 1281-1343.

\bibitem{Tay}
M. Taylor, \emph{Tools for PDE}, Mathematical Surveys and Monographs {\bf 81}, American Mathematical Society, Providence, RI, 2000.


\end{thebibliography}
\end{document}